\pdfoutput=1
\RequirePackage{ifpdf}
\ifpdf 
\documentclass[pdftex]{sigma}
\else
\documentclass{sigma}
\fi

\usepackage{tikz-cd}
\usepackage{nicematrix}

\numberwithin{equation}{section}

\newtheorem{Theorem}{Theorem}[section]
\newtheorem{theoremA}{Theorem}

\newtheorem*{Lemma*}{Lemma}
\newtheorem{Corollary}[Theorem]{Corollary}
\newtheorem{Lemma}[Theorem]{Lemma}
\newtheorem{Proposition}[Theorem]{Proposition}

 { \theoremstyle{definition}
\newtheorem{Definition}[Theorem]{Definition}

\newtheorem{Example}[Theorem]{Example}
\newtheorem{Remark}[Theorem]{Remark} }

\newcommand{\cB}{\ensuremath{\mathcal{B}}}

\newcommand{\cG}{\ensuremath{\mathcal{G}}}
\newcommand{\cH}{\ensuremath{\mathcal{H}}}

\newcommand{\cV}{\ensuremath{\mathcal{V}}}

\newcommand{\R}{\ensuremath{\mathbb{R}}}

\newcommand{\N}{\ensuremath{\mathbb{N}}}

\newcommand{\Bis}{\ensuremath{\operatorname{Bis}}}

\DeclareMathOperator{\Diff}{Diff}

\DeclareMathOperator{\Fl}{Fl}
\DeclareMathOperator{\id}{id}

\begin{document}


\newcommand{\arXivNumber}{2411.00587}

\renewcommand{\PaperNumber}{037}

\FirstPageHeading

\ShortArticleName{The Stacey--Roberts Lemma for Banach Manifolds}

\ArticleName{The Stacey--Roberts Lemma for Banach Manifolds}

\Author{Peter KRISTEL~$^{\rm a}$ and Alexander SCHMEDING~$^{\rm b}$}

\AuthorNameForHeading{P.~Kristel and A.~Schmeding}

\Address{$^{\rm a)}$~Cyberagentur, Gro{\ss}e Steinstra{\ss}e 19, 06108 Halle (Saale), Germany}
\EmailD{\href{mailto:pkristel@gmail.com}{pkristel@gmail.com}}
\URLaddressD{\url{https://peterkristel.com/}}

\Address{$^{\rm b)}$~NTNU Trondheim, Alfred Getz' vei 1, Trondheim, Norway}
\EmailD{\href{mailto:alexander.schmeding@ntnu.no}{alexander.schmeding@ntnu.no}}
\URLaddressD{\url{https://www.ntnu.no/ansatte/alexander.schmeding}}

\ArticleDates{Received November 27, 2024, in final form May 07, 2025; Published online May 18, 2025}

\Abstract{The Stacey--Roberts lemma states that a surjective submersion between finite-dimensional manifolds gives rise to a submersion on infinite-dimensional manifolds of smooth mappings by pushforward. This result is foundational for many constructions in infinite-dimensional differential geometry such as the construction of Lie groupoids of smooth mappings. We generalise the Stacey--Roberts lemma to Banach manifolds which admit smooth partitions of unity.
The new approach also remedies an error in the original proof of the result for the purely finite-dimensional setting.}


\Classification{58D15; 58B20; 58B10; 53C05}

\section{Introduction}

Manifolds of mappings are among the most important examples of infinite-dimensional manifolds with far ranging applications from geometric hydrodynamics, to shape analysis and infinite-dimensional Lie theory.
A fundamental property of these manifolds is that a smooth mapping between manifolds induces a smooth mapping between the associated manifolds of mappings, the push-forward.
To illustrate this, consider smooth (finite-dimensional) manifolds $X$, $M$, $N$ and a smooth map $\varphi \colon M \rightarrow N$.
Denote by $C^k(X,M)$ the maps of class $C^k$, $k\in \N \cup \{\infty\}$.
This set can be turned into an infinite-dimensional manifold with the property that the pushforward
\begin{displaymath}
\varphi_\ast \colon \ C^k (X,M) \rightarrow C^k (X,N), \qquad \varphi_\ast (f)= \varphi \circ f
\end{displaymath}
is smooth.
It is now of interest to understand which properties of $\varphi \colon M \rightarrow N$ lift to properties of the pushforward $\varphi_\ast$.
For example, if $\varphi$ is a submersion, an embedding or a proper map, then~\mbox{\cite[Theorem E]{AaGaS20}} establishes sufficient conditions on the manifolds and the differentiability index $k$ for $\varphi_\ast$ to be of the same type. In particular, the following result, originally introduced by \cite[Corollary 5.2]{stacey2013smooth} for generalised calculi on infinite-dimensional manifolds, is known in the literature (see, e.g., \cite[Lemma 2.24]{Sch23}).

\begin{Lemma*}[Stacey--Roberts lemma] Let $\varphi \colon M \rightarrow N$ be a smooth surjective submersion between finite-dimensional manifolds and $X$ a finite-dimensional smooth manifold. Then the pushforward
\begin{equation*}
 \varphi_\ast \colon\ C^\infty (X,M) \rightarrow C^\infty (X,N)
\end{equation*}
is a submersion.
\end{Lemma*}

Recall that manifolds of smooth mappings are infinite-dimensional manifolds modelled on non-Banach locally convex spaces.
In the setting of manifolds of smooth maps, smoothness of the pushforward has to be understood in the sense of Bastiani calculus, while submersions are defined à la Hamilton \cite[Defenition~4.4.8]{Ham82} (see also Appendix \ref{App:Bast}) as maps admitting submersion charts.
This distinguishes the case of spaces of smooth maps from the~case of $C^k$-mappings for $k<\infty$ which are Banach manifolds. We note that for these Banach manifolds an easier proof of a statement analogous to the Stacey--Roberts lemma can be given using section, see \cite{AaGaS20} for a~detailed proof in the case $k<\infty$.

Manifolds of (smooth) mappings occur naturally in areas such as shape analysis or the \mbox{analysis} of PDE via the Ebin--Marsden approach (also cf.\ \cite{Sch23}). In these settings, the Stacey--Roberts lemma becomes a versatile tool for infinite-dimensional differential geometry. To illustrate this, the following list contains an overview of known use cases for the Stacey--Roberts lemma. Namely, the Stacey--Roberts lemma is an important tool in the construction of
\begin{itemize}\itemsep=0pt
\item (infinite-dimensional) Lie groupoids \cite{AaGaS20},
\item Lie rackoids integrating Courant algebroids \cite{LGaW20},
\item bisection groups and infinite-dimensional symmetry groups related to Lie groupoids~\cite{AaS19,AaS19b,Sch20},
\item smooth loop stacks and Hom stacks of orbifolds \cite{RaV18,RaV18b},
\item smooth mapping stacks in derived differential geometry \cite{ste24},
\item the square root velocity transform for shape spaces with values in homogeneous spaces in shape analysis \cite{CaEaS18}.
\end{itemize}
In the setting of Bastiani calculus, the proof of the Stacey--Roberts lemma appeared first in~\cite{AaS19} but as Steffens noted in private conversation with the second author the published proof contains a~gap (a solution to this problem was then proposed in~\cite{ste24}).
In the present article, we remedy this problem and provide a complete and streamlined proof of the Stacey--Roberts lemma. Moreover, we shall establish the following generalised version of the statement.

\begin{theoremA}[Stacey--Roberts lemma for Banach manifolds] \label{TheoremA}
Let $X$ be a $\sigma$-compact manifold and $M$, $N$ be $C^\infty$-paracompact manifolds modeled on Banach spaces.
If $p \colon M \rightarrow N$ is a smooth surjective submersion, then the pushforward
\begin{displaymath} p_\ast \colon\ C^\infty (X,M) \rightarrow C^\infty (X,N),\qquad p_\ast (f)=p\circ f\end{displaymath}
 is a submersion.
\end{theoremA}

Here $C^\infty$-paracompactness means that every open cover of the manifold (as a topological space) admits a subordinate smooth partition of unity. It implies in particular, that the model spaces of the manifolds $M$, $N$ admit smooth norms and it ensures the existence of a local addition on the Banach manifolds. This is a crucial requirement to construct the manifolds of mappings considered here.

In the original proof of the Stacey--Roberts lemma for finite-dimensional target manifolds in~\cite{AaS19}, Riemannian structures on the target manifolds were exploited. These tools are not available on the Banach manifolds which generally admit only weak Riemannian structures (see~\cite{Sch23}). It turns out that an alternative proof can be given using the language of sprays on anchored bundles. As Theorem \ref{TheoremA} allows us to expand the scope to (certain) Banach manifolds, also the mentioned applications of the Stacey--Roberts lemma generalise to this broader regime. At the end of the present article we shall provide examples of applications which generalise essentially by the same proof to the Banach setting.

\section{Preliminaries on manifolds of mappings}
In this section, we recall results on manifolds of smooth mappings from a finite-dimensional manifold $X$ to a smooth (possibly infinite-dimensional) manifold~$M$. These manifolds cannot be modeled on Banach spaces, hence we need the generalisation of calculus known as Bastiani calculus. See Appendix~\ref{App:Bast}
for a brief recollection and references to the literature on calculus on infinite-dimensional manifolds.

Let $X$ be a locally compact (thus finite-dimensional), paracompact smooth manifold, and ${E \rightarrow X}$ a smooth vector bundle (whose fibres are possibly locally convex spaces).
We write~$\Gamma(E)$ for the \emph{set of smooth sections} of $E$.
We write $0 \in \Gamma(E)$ for the zero section.
Let~${K \subseteq X}$.
Then we consider the subspace of smooth \emph{sections supported in}~$K$, i.e.,
\begin{displaymath}
 \Gamma_{K}(E) := \bigl\{ \sigma \in \Gamma(E) \mid \sigma|_{X \setminus K} = 0 \bigr\}.
\end{displaymath}
Topologize $\Gamma_K(E)$ with the compact-open $C^{\infty}$-topology induced by the inclusion $\Gamma_K(E) \subseteq C^{\infty}(M,E)$, cf.\ \cite{HaS17}.
We further define
\begin{displaymath}
 \Gamma_c (E) := \bigcup_{K \subseteq M \text{ compact}} \Gamma_K (E).
\end{displaymath}
As the spaces $\Gamma_K(E)$ are ordered by inclusion of the compact sets $K$, the \emph{space of compactly supported sections} $\Gamma_c(E)$ becomes the locally convex inductive limit of the $\Gamma_K(E)$. The resulting topology on $\Gamma_c(E)$ is indeed a Hausdorff vector topology by \cite[Lemma F.19\,(c)]{glo04}.

To construct charts for $C^{\infty}(X,M)$ around $f \in C^{\infty}(X,M)$, where $M$ is a possibly infinite-dimensional manifold, consider the space of sections of the pullback bundle $f^{*}TM \rightarrow X$. Charts are then constructed with the help of a local addition on $M$.

\begin{Definition}
 Let $M$ be a smooth manifold.
 A \emph{local addition} on $M$ consists of a pair $(U,\Sigma)$, where $U \subset TM$ is an open neighbourhood of the zero section $M \subset TM$, and $\Sigma \colon U \rightarrow M$,
 such that
 \begin{itemize}\itemsep=0pt
 \item $\Sigma$ restricts to the identity on $M \subset TM$,
 \item the map $\theta_\Sigma \colon U \rightarrow M \times M, \theta_\Sigma (v) := (\pi_M(v), \Sigma(v))$ is a diffeomorphism onto an open subset of $M \times M$.
 \end{itemize}
\end{Definition}

To motivate the terminology of local addition, consider the following (trivial) example. Let~$V$ be a locally convex vector space, and set $U = V \times V \cong TV$. The map $\Sigma (v,w) = v+w$ is then a local addition.

{\bf Canonical charts for manifolds of mappings.}
Let $(U,\Sigma)$ be a~local addition on $M$.
We let $U' := \theta_\Sigma (U) \subseteq M \times M$.
Set $O_{f} = \Gamma_c(f^{*}TM) \cap C^{\infty}(X, U)$, and
\begin{equation*}
 O'_{f} = \bigl\{ g \in C^{\infty}(X,M) \mid (f,g)(X) \subseteq U' \bigr\}.
\end{equation*}
One can show that $O_{f}$ is an open subset of $\Gamma_c(f^\ast TM)$. Similar statements hold for $O'_f$ and a~suitable topology on $C^\infty(X,M)$, see \cite[proof of Theorem~1.4]{GaS22} for details.
We then obtain an inverse pair of maps
\begin{align}\label{eq:can_charts}
 \phi_{f}\colon\ O_{f}\rightarrow O'_{f},\quad \tau \mapsto \Sigma \circ \tau, \qquad
 \phi_{f}^{-1}\colon\ O'_{f} \rightarrow O_{f},\quad g \mapsto \theta_{\Sigma}^{-1} \circ (f,g).
\end{align}
By \cite[Theorem~1.4]{GaS22}, these mappings form a smooth atlas for $C^{\infty}(X,M)$.

\begin{Proposition}[{special case of \cite[Theorem~1.4]{GaS22}}]
 Let $X$ be a paracompact finite-dimensional manifold and $M$ be a Banach manifold which admits a local addition.
 Then the canonical charts 
 form a smooth atlas for a manifold structure on $C^\infty (X,M)$. This manifold structure does not depend on the choice of local addition.
\end{Proposition}

The following result is standard (see, e.g., \cite[Section 2]{Sch23}), but a proof for it has, to the best of our knowledge, not been published for the setting (non-compact source manifold and infinite-dimensional manifold as target) considered in the present article. We include its proof for the readers convenience.

\begin{Proposition}\label{prop:smooth_pf}
Let $X$ be a $\sigma$-compact manifold and $M,N$ locally convex manifolds which admit local additions. If $f \colon M \rightarrow N$ is smooth, then the pushforward
\begin{displaymath}f_\ast \colon \ C^\infty (X,M) \rightarrow C^\infty (X,N),\qquad f_\ast (g) =f \circ g,
\end{displaymath}
is smooth.
\end{Proposition}

\begin{proof}
Since smoothness is a local property, we may test it in suitable charts. Let~$\Sigma_M$ and~$\Sigma_N$ be the local additions on $M$ and $N$ respectively. For $g \in C^\infty (X,M)$ and $f \circ g$, we pick the~canonical charts \smash{$\bigl(O'_g,\phi_g\bigr)$} and \smash{$\bigl(O'_{f\circ g},\phi_{f\circ g}\bigr)$} with respect to the choice of local additions. Then the pushforward becomes in these canonical charts
\begin{displaymath}\phi_{f\circ g} \circ f_* \circ \phi^{-1}_f \colon\ \Gamma_{g,c} \rightarrow \Gamma_{f\circ g, c},\qquad \tau \mapsto \bigl(\theta_{\Sigma_N}^{-1} \circ f \circ \Sigma_M\bigr)_*(\tau).\end{displaymath}
Now the mapping $\theta_{\Sigma_N}^{-1} \circ f \circ \Sigma_M \colon g^{\ast}TM \supseteq g^\ast U_{\Sigma_M} \rightarrow (f\circ g)^\ast TN$ is a smooth map defined on an open neighborhood of the zero section, takes the zero section to itself and descends to the identity. Hence the smoothness of the pushforward follows from Gl\"ockner's $\Omega$-lemma with parameters \cite[Theorem~F.23]{glo04} which is applicable due to \cite[Lemma F.19\,(c)]{glo04} since $X$ is $\sigma$-compact.
\end{proof}

\section{Geometric constructions on Banach manifolds}

In this section, we consider the target manifolds of the manifolds of mappings. To generalise the Stacey--Roberts lemma we need geometric tools on Banach manifolds.
We start with a~section on connections and related objects on Banach manifolds.
While all of these results are standard in the finite-dimensional setting, it is worth pointing out the validity of these results on Banach manifolds which are $C^\infty$-paracompact.
This means that the manifolds are paracompact as topological spaces, and every open covering admits a subordinate smooth partition of unity. In particular, the model Banach spaces thus admit smooth bump functions, see~\cite{BaF66}.
In~infinite dimensions this is a restriction, see \cite[Appendix~A.4]{Sch23} and the references therein. On~$C^\infty$-paracompact Banach manifolds, the typical local-to-global arguments employing smooth partitions of unity are available.

\subsection{Connections on Banach vector bundles}

By a \emph{Banach vector bundle}, we will denote a vector bundle over a Banach manifold whose fibres are all Banach spaces. Basic differential geometry for this setting can, e.g., be found in Lang's book~\cite{La01}.

\begin{Definition}\label{defn:connect_lin}
Let $\pi \colon E \rightarrow M$ be a Banach vector bundle with typical fibre $F$. Denote by $H$ the model space of $M$. A \emph{linear connection} is a bundle map $K \colon TE \rightarrow E$ over $\pi$
\begin{displaymath}\begin{tikzcd}
TE \arrow[r,"K"]\arrow[d,"\pi_E"] & E \arrow[d,"\pi"]\\
E \arrow[r,"\pi"] & M
\end{tikzcd}\end{displaymath}
such that in bundle trivialisations $(U,\Phi)$ of $E$ and $(TU,T\Phi)$ of $TE$ we have \begin{displaymath}\Phi \circ K\circ T\Phi(u,\xi, X, \eta) = (u,\eta + B_U (u)(X,\xi)),\end{displaymath} where $B_U \colon U \rightarrow \mathrm{Bil} (H,F;F)$ is smooth as a map into the continuous bilinear mappings, \mbox{\cite[Definition 1.5.9]{Kli95}}, \begin{displaymath}
 \mathrm{Bil}(H,F;F) := \{\beta \colon H \times F \rightarrow F \mid \beta \text{ continuous and bilinear}\}.
\end{displaymath}
\end{Definition}

\begin{Remark}\label{rem:ex_lincon}
Every Banach vector bundle over a $C^\infty$-paracompact base admits linear connections, \cite[Theorem 1.5.15]{Kli95} (see also \cite{Eli67} for applications for manifolds of mappings).

By smooth paracompactness of the base, the existence of a linear connection on $TM$ is equivalent to the existence of a covariant derivative and the existence of a spray on $M$ (see \mbox{\cite[Section VIII]{La01}} or \cite[Section 1.5]{Kli95} for details).
\end{Remark}

We have a notion of parallelism with respect to the chosen linear connection.
\label{setup:parallel_tp}
Let $\pi \colon E \rightarrow M$ be a vector bundle with a linear connection $K$ and $E_m := \pi^{-1}(m)$.
Then for every $\gamma \colon [0,1] \rightarrow M$ there is a unique map $P^\gamma\colon \gamma^\ast E \rightarrow \gamma^\ast E$ called \emph{parallel translation}.
Localising on an open set $U$ in a bundle trivialisation (which we suppress), parallel transport is the unique solution $\eta$ to the linear differential equation (see, e.g.,~\mbox{\cite[Lemma~1.6.14]{Kli95}})
\begin{align}\label{eq:PT_diffeo}
\dot{\eta}(t) + B_U (\gamma(t)) (\dot{\gamma}(t),\eta(t))=0,
\end{align}
where $B_U$ is the smooth map induced by $K$ as in Definition~\ref{defn:connect_lin}.
In particular, for every $t,s \in [0,1]$ the assignment
$P_{s,t}^\gamma \colon E_{\gamma (s)} \rightarrow E_{\gamma (t)}$
is a Banach space isomorphism.

We come now to a type of connection which is (in general) not a linear connection.

\begin{Definition}
Let $p \colon M \rightarrow N$ be a surjective submersion between $C^\infty$-paracompact Banach manifolds (here submersion is to be understood in the sense of Definition~\ref{defn:submersion}). Recall the~following standard definitions:
\begin{enumerate}\itemsep=0pt
\item The \emph{vertical bundle} $\mathcal{V} \subseteq TM$ is the vector subbundle fibre-wise given by $\ker T p$.
\item An \emph{Ehresmann connection} for $p$ is a smooth subbundle $\mathcal{H} \subseteq TM$, called the \emph{horizontal bundle} which complements $\cV$, i.e., $TM=\mathcal{V} \oplus \mathcal{H}$.
Equivalently, an Ehresmann connection is a splitting of the short exact sequence of vector bundles
\begin{displaymath}0 \rightarrow \cV \rightarrow TM \xrightarrow{(\pi_M,{\rm d}p)} p^{\ast}TN \rightarrow 0.\end{displaymath}
Here the map $(\pi_M,{\rm d}p) \colon TM \rightarrow p^\ast TN$ is the bundle map over the identity given fibre wise by the derivative $T_xp \colon T_x M \rightarrow T_{p(x)}N$ for $x\in M$.
\end{enumerate}
\end{Definition}

\begin{Lemma}\label{lem:ex_Ehresmann}
Let $M$ be a $C^\infty$-paracompact manifold and $p\colon M \rightarrow N$ be a surjective submersion. Then there exists an Ehresmann connection $\cH$ for $p$.
\end{Lemma}

\begin{proof}
 Pick a family of submersion charts $(U_\alpha, \psi_\alpha,\varphi_\alpha)_{\alpha}$ for $p$ such that the domains $U_\alpha$ of $\psi_\alpha$ cover $M$.
 We obtain commutative squares
 \begin{equation*}
 \begin{tikzcd}
 M \supseteq U_{\alpha} \ar[r,"\psi_{\alpha}"] \ar[d,"p", shift left=12pt] & U'_{\alpha}\subseteq E \times F \ar[d] & & TM \supseteq TU_{\alpha} \ar[r,"T\psi_{\alpha}"] \ar[d,"Tp", shift left=12pt] & U_{\alpha}' \times (E \times F) \ar[d] \\
 N \supseteq V_{\alpha} \ar[r,"\varphi_{\alpha}"] & V'_{\alpha} \subseteq E, & & TN \supseteq TV_{\alpha} \ar[r,"T\varphi_{\alpha}"] & V'_{\alpha} \times E
 \end{tikzcd}
 \end{equation*}
 and
 \begin{equation*}
 \begin{tikzcd}
 p^{*}TN\supseteq p^{*}TV_{\alpha} \ar[d, shift left=10pt] \ar[r, dashed,"\hat{\psi}_{\alpha}"] & U'_{\alpha} \times E \ar[d] \\
 M \supseteq U_{\alpha} \ar[r,"\psi_{\alpha}"] & U'_{\alpha}.
 \end{tikzcd}
 \end{equation*}
 Where the dashed arrow is given by
 \begin{align*}
 p^{*}TV_{\alpha} \subseteq U_{\alpha} \times TV_{\alpha} \rightarrow U'_{\alpha} \times E,\qquad \hat{\psi}_{\alpha}(u,v) := (\psi_{\alpha}(u), \mathrm{pr}_{2}(T\varphi_{\alpha}(v))).
 \end{align*}
 We thus see that a splitting, denoted by $\sigma_{\alpha}$, of the sequence
 \begin{equation*}
 \begin{tikzcd}
 0 \ar[r] & \mathcal{V} \ar[r] & TU_{\alpha} \ar[r] & p^{*}TV_{\alpha} \ar[r] \ar[l,bend right, dashed, "\sigma_{\alpha}"'] & 0
 \end{tikzcd}
 \end{equation*}
 is given by
 \smash{$
 p^{*}TV_{\alpha} \xrightarrow{\hat{\psi}_{\alpha}} U'_{\alpha} \times E \hookrightarrow U'_{\alpha} \times (E \times F) \xrightarrow{(T\psi_{\alpha})^{-1}} TU_{\alpha}.
 $}
 Let $h_{\alpha}$ be a smooth partition of unity subordinate to the family $U_{\alpha}$.
 We observe that the map
 \begin{equation*}
 \hat{\sigma}_{\alpha}(u,v) := \begin{cases}
 h_{\alpha}(u) \sigma_{\alpha}(v) & \text{if } u \in U_{\alpha}, \\
 0 & \text{if } u \notin U_{\alpha}
 \end{cases}
 \end{equation*}
 is a smooth map $p^{*}TN \rightarrow TM$.
 It follows that $\sum_{\alpha} \hat{\sigma}_{\alpha}$ is an Ehresmann connection.
\end{proof}

\subsection{Sprays on anchored bundles}
For our construction of manifolds of mappings, local additions are needed. Since Banach manifolds in general do not admit a natural Riemannian manifold structure (whose exponential map can be used as a local addition), another source for local additions is needed. Sprays and their exponential mappings \cite{La01}, provide a convenient way to replace the missing structure. We need a generalisation of the classical sprays to sprays of anchored Banach bundles, see \cite{ana11,CaP12,CaMaS20}.

\begin{Definition}
Let $\pi \colon E \rightarrow M$ be a Banach vector bundle with typical fibre $F$. Together with a~vector bundle morphism $\rho \colon E \rightarrow TM$, the \emph{anchor}, we call the triple $(\pi,E,\rho)$ an \emph{anchored Banach bundle $($over $M)$}.
\end{Definition}

\begin{Example}\label{ex:subbun_anchor}
Every vector subbundle of $TM$ becomes an anchored bundle with respect to the trivial anchor given by inclusion (cf.\ \cite[Example~3.2]{CaP12}).
If $p\colon M \rightarrow N$ is a surjective submersion, there are several anchored bundles associated to~$p$. The vertical bundle $\cV$ and for every Ehresmann connection the horizontal bundle $\cH$ are subbundles, hence anchored. Note that an Ehresmann connection described as the section $\sigma_\cH$ of $(\pi_M,{\rm d}p)$ turns
\begin{displaymath}\sigma_{\cH} \colon\ p^\ast TN \rightarrow \cH \subseteq TM
\end{displaymath}
into an anchored bundle.
\end{Example}

Anchored Banach bundles are related to, but weaker structures than Banach Lie algebroids (see, e.g., \cite{AaGaS20, BaGaP19}). The anchor allows us to generalise second order differential equations to vector bundles as follows.

\begin{Definition}\label{defn:spray}
Let $(\pi, E, \rho)$ be an anchored Banach vector bundle.
A smooth map $S\colon E \rightarrow TE$ on $E$ is called a \emph{semispray} if it satisfies
\begin{itemize}\itemsep=0pt
\item[\hypertarget{sprayS1}{(S1)}] $\pi_{E} \circ S = \id_E$, where $\pi_E \colon TE\rightarrow E$ is the projection (i.e., $S$ is a vector field).
\item[\hypertarget{sprayS2}{(S2)}] $S$ \emph{lifts the anchor} in the sense that $T\pi \circ S = \rho$.
\end{itemize}
For every $\lambda \in \R$ we let $h_\lambda \colon E \rightarrow E$ be the vector bundle morphism which in every fibre $E_x$ is given by multiplication with $\lambda$.
If a semispray satisfies the condition
\begin{itemize}\itemsep=0pt
\item[\hypertarget{sprayS3}{(S3)}] $S\circ h_\lambda = Th_\lambda (\lambda S), \qquad \forall \lambda \in \R,$
\end{itemize}
then we call $S$ a \emph{spray} on the anchored Banach bundle.
\end{Definition}

\begin{Example}
For the trivial anchored bundle $(\pi_M, TM,\id_{TM})$, a spray as per Definition \ref{defn:spray} is just a spray in the usual sense of differential geometry \cite[Chapter~IV, Section~3]{La01}.
In particular, if~$(M,g)$ is a strong Riemannian manifold \cite[Section~4]{Sch23}, then $TM$ admits a spray, called the \emph{metric spray} compatible with the Riemannian metric~$g$.
\end{Example}

Since sprays are vector fields, the notion of relatedness (via a smooth map) makes sense for them. We will say that sprays $S_1$, $S_2$ are $\varphi$\emph{-related} if they are $\varphi$-related as vector fields, i.e., $T\varphi \circ S_1 = S_2 \circ \varphi$ holds.
Given a vertical isomorphism $\varphi$, the next Lemma~\ref{lem:anchoredIso} constructs a~$\varphi$-related spray to a given spray.

\begin{Lemma}\label{lem:anchoredIso}
 Let $(\pi_{i},E_{i},\rho_{i})$ be anchored Banach vector bundles on $M$.
 Let $\varphi \colon E_{1} \rightarrow E_{2}$ be a~$($vertical$)$ isomorphism of anchored Banach bundles on~$M$, i.e., $\pi_{1} = \pi_{2}\circ \varphi$ and $\rho_{1} = \rho_{2} \circ \varphi$.
 If $S\colon E_{1} \rightarrow TE_{1}$ is a spray, then $T\varphi \circ S \circ \varphi^{-1}\colon E_{2} \rightarrow TE_{2}$ is a spray.
 Moreover, $S$ is $\varphi$-related to $T \varphi \circ S \circ \varphi^{-1}$.
\end{Lemma}

\begin{proof}
 We compute
 \begin{align*}
 &\pi_{2} \circ T\varphi \circ S \circ \varphi^{-1}= \varphi \circ \pi_{1} \circ S \circ \varphi^{-1} = \id_{E_{2}}, \qquad\text{and}\\
 &T\pi_{2} \circ T\varphi \circ S \circ \varphi^{-1}= T\pi_{1} \circ S \circ \varphi^{-1} = \rho_{1} \circ \varphi^{-1} = \rho_{2},
 \end{align*}
 and finally
 \begin{align*}
 T\varphi \circ S \circ \varphi^{-1} \circ h_{\lambda} &{}= T\varphi \circ S \circ h_{\lambda} \circ \varphi^{-1} = T \varphi \circ Th_{\lambda} (\lambda S) \circ \varphi^{-1} \\
 &{}= Th_{\lambda} \circ T\varphi \circ (\lambda S) \circ \varphi^{-1} = Th_{\lambda} \circ \lambda\bigl(T\varphi \circ S \circ \varphi^{-1}\bigr). \qedhere
\tag*{\qed}
\end{align*}
\renewcommand{\qed}{}
\end{proof}

For bundles over a $C^\infty$-paracompact base, a local to global argument yields the following.

\begin{Lemma}\label{lem:ex_spray_anchor}
Let $(\pi,E,\rho)$ be an anchored Banach bundle on a manifold $M$ which is $C^\infty$-paracompact, then there exists a spray on the anchored bundle.
\end{Lemma}

\begin{proof}
 Let $F$ be the typical fiber of the vector bundle $E \rightarrow M$.
 Let $\{U_{i}, \varphi_{i}\}_{i\in I}$ be a trivializing cover for $E \rightarrow M$, i.e.,~$U_{i} \subseteq M$ open, and \smash{$\varphi_{i}\colon E|_{U_{i}} \xrightarrow{\cong} U_{i} \times F$}.
 For every $i\in I$, we thus obtain an anchored Banach bundle $A_{i} := \bigl(\mathrm{pr}_{1},U_{i} \times F, \rho \circ \varphi_{i}^{-1}\bigr)$.
 Pick a continuous and symmetric bilinear map $B\colon F \times F \rightarrow F$.
 For every $i$, the map
 \begin{align*}
 S_{i}\colon\ U_{i} \times F \rightarrow TU_{i} \times F \times F \cong T(U_{i} \times F), \qquad
 S_i(x,v) = \bigl(\rho \bigl(\varphi_{i}^{-1}(x,v)\bigr),v,B(v,v)\bigr)
 \end{align*}
 defines a spray on $A_{i}$.

 By Lemma \ref{lem:anchoredIso}, the spray $S_{i}$ on $A_{i}$ yields a spray on $(\pi, E|_{U_{i}},\rho)$.
 By using a smooth partition of unity, one then obtains a globally defined section $S\colon E \rightarrow TE$.
 That this defines a~spray follows from the fact that the set of sprays is a convex subset of the set of sections of~${TE \rightarrow E}$.
\end{proof}

Anchored bundles are a setting in which one can make sense of sprays and differential equations whose right-hand side is given by a spray.
Other natural examples of sprays on anchored bundles arise in the theory of Lie groupoids and Lie algebroids, see \cite{CaMaS20,CaF03}.
The following definition builds on the usual proof for the existence of solutions to ordinary differential equations. In~the~classical settings of sprays on the tangent bundle the proofs are recorded in~\mbox{\cite[Chapter~IV, Section~4]{La01}}, the proofs easily adapt to the more general situation.

\begin{Definition}
For a spray $S$ on an anchored bundle $(\pi,E,\rho)$ over $M$, we denote its \emph{flow} (in the sense of flows of vector fields) by{\samepage
\begin{displaymath}
 \Fl^S \colon\ \R \times E \supseteq D \rightarrow E,\qquad (t,v) \mapsto \Fl_t^S (v),\end{displaymath} with $\{0\}\times E \subseteq D\subseteq \R \times E$ open, $\partial_t \Fl_{t}^{S}(v) = S\bigl(\Fl_{t}^{S}(v)\bigr)$ and $\Fl_0^S(v)=v$.}

Since $S$ vanishes on the zero section $0_E$ by (\hyperlink{sprayS3}{S3}), there exists an open neighborhood $\Omega \subseteq E$ of $0_E$ in $E$ such that $\Omega \cap \pi^{-1}(x)$ is convex for every $x \in M$ and the flow is defined up to time~$1$. Hence one defines the \emph{spray exponential map}
\begin{displaymath}\exp_S \colon\ \Omega \rightarrow M,\qquad \exp_S = \pi \circ \Fl^S_1.\end{displaymath}
\end{Definition}

Let us now state some properties of the flow curves of a spray. For this we define \emph{anchored paths} which in the Lie algebroid literature are also called $A$-paths (or admissible paths), see, e.g., \cite[\S1.1]{CaF03}.

\begin{Definition}
Let $(\pi,E,\rho)$ be an anchored bundle over $M$. A $C^1$-curve $c \colon I \rightarrow E$ on a~non-degenerate interval $I$ is called a $(\rho-)$\emph{anchored path} if for the base path $\pi (c(t))\colon I \rightarrow M$ the~following identity holds
\begin{displaymath}\rho(c(t))=\frac{{\rm d}}{{\rm d}t} \pi (c(t)), \qquad \forall t \in I.
\end{displaymath}
\end{Definition}

\begin{Lemma}\label{lem:spec_curv}
Let $(\pi, E,\rho)$ be an anchored bundle over $M$ with a spray $S$. Then
\begin{itemize}\itemsep=0pt
\item[$(a)$] every integral curve of $S$, i.e.,~$c \colon I \rightarrow E$, $\frac{{\rm d}}{{\rm d}t}c(t)=S(c(t))$, is an anchored path,
\item[$(b)$] if for $x \in M$ we denote by $\exp_S^x \colon \pi^{-1}(x)\rightarrow M$ the restriction of $\exp_S$ to the fibre, then $T_{0_x} \exp_S^x = \rho|_{\pi^{-1}(x)}^{T_xM}$.
\end{itemize}
\end{Lemma}

\begin{proof}
(a) For $v \in E$, let $I_v$ be the interval on which $\beta_v (t)= \Fl_{t}^S(v)$ is defined. Assume that~${s,t \in \R}$, such that for $v \in E$ the interval $I_v$ contains $st$.
The usual proof for flows of ordinary differential equations (see \cite[p.~106]{La01} which carries over) shows that this happens if and only if~${s \in I_{\beta_v(t)}}$ and then
\begin{align}
\pi\Fl_{s}^S(tv)= \pi \Fl^S_{st}(v).\label{flow:prop}
\end{align}
By construction $\beta_v$ is the integral curve of $S$ starting at $v$.
We compute the derivative of $\beta_v$ via the chain rule and (\hyperlink{sprayS2}{S2}).
Using {$\beta_v (t)=\Fl^S_t (v)$} it follows that
\begin{displaymath}
 \partial_t \bigl(\pi\bigl(\Fl^S_t (v)\bigr)\bigr) = T\pi \circ S \bigl(\Fl^S_t (v)\bigr) = \rho \bigl(\Fl^S_t (v)\bigr).\end{displaymath}
Hence integral curves are anchored paths.
(b) The classical proof from~\cite{La01} works. Compute the fibre derivative of the $\exp_S$ along a ray $t\mapsto tv \in \pi^{-1}(x) \subseteq E$ by using the chain rule and~\eqref{flow:prop} to obtain
\begin{align*}
T_{0_x}\exp_S (v)&=\left.\frac{{\rm d}}{{\rm d}t}\right|_{t=0} \exp_S (tv)= \left. \frac{{\rm d}}{{\rm d}t}\right|_{t=0} \pi \Fl^{S}_{1}(tv) = \left. \frac{{\rm d}}{{\rm d}t}\right|_{t=0} \pi \Fl^{S}_{t}(v) \\
&=T\pi \circ S\bigl(\Fl_0^S(v)\bigr) =\rho (v),
\end{align*}
where we used (\hyperlink{sprayS2}{S2}).
\end{proof}

Stronger results along the lines of Lemma \ref{lem:spec_curv} are recorded in \cite{ana11}, however we do not have need for them.

 Obviously, if $\exp_S$ is locally invertible at the zero section in $E$ then $\exp_S^x$ needs to restrict to an isomorphism.
 Hence if the anchor is not a bundle isomorphism, $\exp_S$ cannot be locally invertible.
 For examples of this behaviour, consider the sprays on one of the subbundles of the tangent bundle from Example \ref{ex:subbun_anchor}.
 This is in contrast to the following observation for the classical situation recorded in Corollary \ref{cor:exp_locadd} below.

There is a theory of connections for sprays on Lie algebroids, see, e.g., \cite{CaF03} and the references therein. Sprays on anchored bundles induce linear connections similar to sprays on the tangent bundle. However, we will not need this theory here.

\subsection{Sprays and submersions}
In this section, we prepare some constructions of special sprays for the proof of the Stacey--Roberts lemma.
Let $p\colon M \rightarrow N$ be a submersion, and $\cV = \ker Tp$ be the vertical subbundle.
The following definition is due to \cite[Definition 3.2.15]{ste24}.

\begin{Definition}
A \emph{local addition on $p$} is a pair $U \subseteq \cV$ open and containing the zero section together with an open embedding $\Psi \colon U \rightarrow M \times_p M$ into the fibre product such that the following diagram commutes:
\begin{equation}\label{diag:loc_add_on_sub}
\begin{tikzcd}
M \arrow[r,hook,"\Delta"] \arrow[d,"0"]& M \times_p M \arrow[d,"\text{pr}_1"] \\
U \arrow[ru,"\Psi"] \arrow[r,"\pi_{\cV}"] & \,M,
\end{tikzcd}
\end{equation}
where $\pi_{\cV} \colon \cV \rightarrow M$ is the basepoint projection, $\Delta$ the diagonal embedding and $\text{pr}_1$ denotes the projection onto the first coordinate in the fibre-product.
\end{Definition}

We remark that a local addition on a submersion is a weaker notion than a local addition on a manifold which respects the fibres of a submersion. The stronger notion has been used in constructions of manifolds of mappings (see, e.g., the recent \cite[Section~5.9]{Mic20} and the references therein). Before we establish the existence of a local addition on submersions, we will prove an auxiliary lemma.

\begin{Lemma}\label{lem:loc2glob_emb}
Let $\pi \colon E\rightarrow M$ be a Banach vector bundle and $q \colon M \rightarrow N$ be a submersion to a Banach manifold. Assume that $M$ is paracompact and $\Phi := (\pi,\phi) \colon E \supseteq \Omega \rightarrow M \times_q M$ is a smooth map such that for every $m\in M$ the map $T_{0_m} \Phi \colon T_{0_m}E \rightarrow T_{\Phi(0_m)}M \times_q M$ is an isomorphism of Banach spaces. Then there is an open neighborhood of the zero section $O\subseteq \Omega$ such that $\Phi|_O$ is an open embedding.
\end{Lemma}
\begin{proof}
As $T_{0_m}\Phi$ is an isomorphism of Banach spaces, the inverse function theorem implies that for every $m \in M$ there is an open neighbourhood $0_m \in U_{m} \subseteq E$ on which $\Phi$ is an open embedding.
By projecting down, we obtain an open cover of $M$:
$V_{m} := \pi(U_{m})$, $m \in M$.
Without loss of generality, we may assume that there exists $\kappa_m \colon \pi^{-1}(V_{m}) \rightarrow V_m \times F$ a bundle trivialisation.
Since $\kappa_{m}(U_{m})$ is an open neighbourhood of $(m,0) \in V_{m} \times F$, we can shrink~$V_{m}$ to find an open neighbourhood $0 \in O_{m}$ such that $V_{m} \times O_{m} \subseteq \kappa_{m}(U_{m})$.
If we replace $U_{m}$ by~$\kappa_{m}^{-1}(V_{m} \times O_{m})$ we obtain the following diagram
\begin{equation*}
 \begin{tikzcd}
 0_{V_{m}} \arrow[r, phantom, sloped, "\subseteq"] \ar[d,"\kappa_{m}", "\cong"'] & U_{m} \arrow[r, phantom, sloped, "\subseteq"] \ar[d,"\kappa_{m}", "\cong"'] & \pi^{-1}(V_{m}) \ar[d,"\kappa_{m}", "\cong"'] \\
 V_{m} \times \{ 0 \} \arrow[r, phantom, sloped, "\subseteq"] & V_{m} \times O_{m} \arrow[r, phantom, sloped, "\subseteq"] & V_{m} \times F.
 \end{tikzcd}
\end{equation*}
Repeating for every $m \in M$, we obtain an open cover $\{ V_{m} \}_{m \in M}$ and by paracompactness of $M$, there is a locally finite refinement of this cover, denote this cover by $\{V_i\}_{i\in I}$.

Now, let $m \in M$ be arbitrary.
That $\{V_i\}_{i\in I}$ is locally finite means that there exists an open neighbourhood of $m$, say $V'_{m} \subseteq M$, which has non-empty intersection with only finitely many of the sets $V_{i}$ for $i \in I$.
Denote by $J_{m}\subset I$ the corresponding finite subset.
I.e.,~$J_{m}$ is defined to be the subset of $I$ such that, if $i \in I$, then
$V_{i} \cap V'_{m} \neq \emptyset \Longleftrightarrow i \in J_{m}.$
Define the open subsets
$
 O'_{m} := \cap_{i \in J_{m}} O_{i} \subseteq F,
 $ and
\begin{equation*}
 Q_{m} := \bigcup_{i \in J_{m}} \kappa_{i}^{-1}\bigl(\bigl(V_{i}\cap V_{m}'\bigr) \times O_{m}'\bigr) = \pi^{-1}\bigl(V_{m}'\bigr) \cap \biggl(\bigcup_{i \in J_{m}} \kappa_{i}^{-1}\bigl(V_{i} \times O_{m}'\bigr)\biggr).
\end{equation*}
We observe that $Q_{m}$ has the following properties:
\begin{itemize}\itemsep=0pt
 \item[(i)] $0_{m} \in Q_{m}$, because $m$ must be contained in at least one of the sets $V_{i} \cap V_{m}'$ for $i \in J_{m}$, and~$0 \in O_{m}'$,
 \item[(ii)] $\pi(Q_{m}) = V_{m}'$,
 \item[(iii)] let $m_{0} \in V_{m}'$ be arbitrary, if $m_{0} \in V_{i}$, then
 \begin{equation*}
 \mathcal{V}_{m_{0}} \cap Q_{m} = \mathcal{V}_{m_{0}} \cap \biggl( \bigcup_{i \in J_{m}} \kappa_{i}^{-1}\bigl(V_{i} \times O'_{m}\bigr)\biggr) \subseteq U_{i}.
 \end{equation*}
\end{itemize}
By (i), it follows that
$
 Q := \bigcup_{m \in M} Q_{m}
$ is an open neighbourhood of $0 \subseteq E$.
We claim that $\Phi$ restricts to an open embedding on $Q$.
Indeed, as $\Phi$ is already locally an open embedding on~$Q$, it suffices to prove that $\Phi$ is injective on $Q$.
So, let $x,y \in Q$ and assume that $\Phi(x) = \Phi(y)$.
It~follows from $\Phi=(\pi,\phi)$ that $\pi(x) = \pi(y) = m_{0}$, and thus $x,y \in \pi^{-1}(m_{0})$.
We have $x \in Q_{m_{x}}$ and $y \in Q_{m_{y}}$ for some $m_{x},m_{y} \in M$.
And thus, by~(ii) we have $m_{0} \in V_{m_{x}}' \cap V_{m_{y}}'$.
By (iii), we have that $x,y \in U_{i}$, whenever $m_{0} \in V_{i}$.
But we already knew that $\Phi$ is an embedding on $U_i$, and therefore injective and thus $x=y$.
\end{proof}

As a first application of the lemma we record the following well known fact.

\begin{Corollary}\label{cor:exp_locadd}
Every $C^\infty$-paracompact Banach manifold $M$ admits a local addition.
\end{Corollary}
\begin{proof}\hspace{-1pt}As $M$ is $C^\infty$-paracompact, the trivial anchored bundle $(\pi_M,TM, \id_{TM})$ admits a~spray~$S$ by Lemma~\ref{lem:ex_spray_anchor}. The spray exponential is smooth and satisfies $\exp_S (0_m)=m$. Computing for $m\in M$ in charts (which we suppress), we can identify $T_{0_x} (\pi_M,\exp_S)$ with the block operator \smash{$\bigl[\begin{smallmatrix} \id & 0 \\ \star & T_{0_x} (\exp_S|_{T_mM})\end{smallmatrix}\bigr]$}. This operator is invertible as Lemma~\ref{lem:spec_curv}\,(b) implies that $T_{0_x} (\exp_S|_{T_mM})$ is the restriction of the anchor of the trivial anchored bundle~$TM$, i.e., the identity. As $M\times M = M\times_{\star} M$, where $M\rightarrow \{\star\}$ is the submersion onto the one-point manifold, we see that $(\pi_M,\exp_S)$ induces an embedding on some open neighborhood of the zero section by Lemma~\ref{lem:loc2glob_emb}. This finishes the proof.
\end{proof}

\begin{Lemma}\label{lem:con_loc_add_subm}
Let $p\colon M \rightarrow N$ be a surjective submersion between Banach manifolds.
\begin{enumerate}\itemsep=0pt
\item[\textup{(a)}] Let $S$ be a spray on $\cV$.
Then, for every $x \in N$, there exists a unique spray ${S_{x}\colon\! TM_{x}\! \rightarrow\! T^{2}M_{x}}$ on the submanifold $M_x := p^{-1}(x)$ such that $S$ and $S_x$ are $T\iota_x$-related, where $\iota_x \colon M_x \rightarrow M$ is the embedding.
\item[\textup{(b)}] If $M$ is $C^\infty$-paracompact then there exists a local addition on $p$.
\end{enumerate}
\end{Lemma}

\begin{proof} We view $(\pi_{\cV},\cV,\rho)$ as an anchored bundle by considering the bundle inclusion ${\rho \colon\! \cV \!\rightarrow\! TM}$ as anchor. Since $p$ is a submersion, $M_x= p^{-1}(x)$ is an embedded submanifold of $M$ such that fibre-wise $T_m M_x = \cV_m$ for each $m \in M_x$. We let $\iota_x \colon M_x \rightarrow M$ be the canonical embedding and observe that $T\iota_x$ factors through the anchor $\rho$, so we abuse notation and view $T\iota_x$ as a~map~to~$\cV$.

(a) Let $S$ be a spray on $(\pi_\cV, \cV, \rho)$. Since $T\pi_\cV \circ S = \rho$ is the bundle inclusion and fibre-wise $T\cV_m = T(T_m M_x)$ for $m \in M_x$, we see that $S$ factors through the image of $T^2 \iota_x$.
As $\iota_x$ is an~embedding, so is $T^2\iota_x$ hence there is a unique smooth vector field
\begin{displaymath}S_x \colon\ T M_x \rightarrow T^2M_x\end{displaymath}
to which $S$ is $T\iota_x$ related. Copying the proof of Lemma~\ref{lem:anchoredIso}, $S_x$ is also a spray.

(b) As $M$ is $C^\infty$-paracompact, there exists a spray $S_\cV$ on the anchored bundle $(\pi_\cV,\cV,\iota_{\cV})$ by Lemma~\ref{lem:ex_spray_anchor}.
The exponential map $\exp_{S_\cV}$ of this spray then yields a local addition on $p$ by restricting the map $\Psi := (\pi_{\cV}, \exp_{S_\cV})$ to a suitable open subset $U \subseteq \mathcal{V}$. To prove this, we let~${U_{0} \subseteq \mathcal{V}}$ be the domain of definition of $\exp_{S_{\cV}}$.
Note first that $\Psi \colon U_{0} \rightarrow M \times_p M$ makes sense as a~map to $M \times_p M$: By part~(a), the spray $S_\cV$ induces a spray $S_x$ on every fibre $M_x$ which is related to $S_\cV$ via the (tangent of the) inclusion of the submanifold. In particular, integral curves of the induced spray stay on~$M_x$ and coincide by relatedness of the vector fields to integral curves of $S_\cV$. Since $\exp_{S_\cV}$ follows these integral curves, $\Psi= (\pi_{\cV},\exp_{S_\cV})$ is a smooth map to $M\times_p M$.

We conclude from the construction of $\Psi$ that it fits into the diagram \eqref{diag:loc_add_on_sub} and makes it commutative. By construction, $\Psi$~is defined on an open neighborhood of the zero section and smooth. Indeed the flow on the zero section is stationary, while the right-hand triangle in the diagram commutes trivially for~$\Psi$.

It remains to be shown that $\Psi$ restricts to an open embedding on some neighborhood of the $0$-section. For this let $m \in M$ be arbitrary and $W \subseteq M$ be a sufficiently small open neighbourhood of $m$, so that there exist trivializations $TW \cong W \times E$ and $\cV|_{W} \cong W \times F$. Here $E \cong H\times F$ as~$F$ is a complemented subspace of~$E$. This induces a trivialization $T\cV|_{W} \cong \cV|_{W} \times E \times F$. To~compute $T_{0_m}\Psi$, we identify $T(M\times_p M)\cong TM\times_{Tp} TM\subseteq TM \times TM$ via Lemma~\ref{lem:tangent_fibreprod}. Relative to these trivializations, we obtain
\begin{align}
 T_{0_{m}}\Psi\colon\ E \times F = H\times F \times F &\rightarrow T_{(m,m)}(M \times_{p} M) \subseteq T_{m}M \times T_{m}M \cong (H\times F)^2, \notag \\
 (x,v,w) &\mapsto ((x,v), T_{0_{m}}\exp_{S_{\cV}}((x,v),0) + (0,w)),\label{eq:TG2_deriv}
\end{align}
where we used linearity to split $T_{0_m}\exp$ and identified the fibre derivative as the bundle inclusion of the vertical bundle via Lemma~\ref{lem:spec_curv}. Using that the right-hand side of \eqref{eq:TG2_deriv} is contained in~${T(M\times_p M)=TM \times_{Tp} TM}$, we can invert $T_{0_m}\Psi$ via
\begin{align*}&\Phi\colon\ (H \times F)\times_{T_mp} (H\times F)\rightarrow (H\times F) \times \{0\} \times F \cong E \times F,\\
&\Phi ((x,v),(x,w)):= \bigl((x,v),(x,w)-T_{0_m}\exp_{S_\cV}((x,v),0)\bigr).
\end{align*}
We conclude that $T_{0_m}\Psi$ is invertible for every $m \in M$ and since $M$ is $C^\infty$-paracompact, Lemma~\ref{lem:loc2glob_emb} shows that we can shrink the domain of $\Psi$ to obtain an open neighborhood of the zero section on which $\Psi$ is an embedding.
\end{proof}

Finally, we can lift a spray to a spray on a horizontal bundle for a given submersion.

\begin{Lemma}\label{lem:spray_lifting}
Let $S$ be a spray on a Banach manifold $N$ and $p \colon M \rightarrow N$ a surjective submersion. Assume that $\cH$ is the horizontal bundle given by an Ehresmann connection. For the anchored bundle $(\pi_H,\cH, \rho)$ with $\rho \colon \cH \rightarrow TM$ being the inclusion, there exists a spray $S_\cH$ which is $Tp|_{\cH}$-related to the spray $S$.
As a consequence on their domain of definitions the following relation holds
\begin{displaymath} p\circ \exp_{S_\cH} = \exp_S \circ Tp|_{\cH}.\end{displaymath}
\end{Lemma}

\begin{proof}
Following standard notation conventions, consider the following commutative diagram:
\begin{displaymath}
\begin{tikzcd}
 T(p^{\ast} TN) \arrow[r, "\pi_{p^\ast TN}"] \arrow[d, "T(\pi_N^\ast p)"] & p^\ast TN \arrow[r,"p^\ast \pi_N"] \arrow[d,"\pi_N^\ast p"] & M \arrow[d,"p"] \\ T^2 N \arrow[r,"\pi_{TN}"] & TN \arrow[r,"\pi_N"] & \,N.
\end{tikzcd}
\end{displaymath}
We let $\pi_N^\ast p \colon p^\ast TN \rightarrow TN$ be the bundle morphism induced by the pullback bundle.

We now claim that there exists a vector field $\hat{S}\colon p^{*}TN \rightarrow T(p^{*}TN)$, which is $\pi^{*}_{N}p$-related to $S\colon TN \rightarrow T^{2}N$.
We identify via Lemma~\ref{lem:tangent_fibreprod}
\begin{align*}
 & p^{*}TN= M \, {}_{p}\hspace{-.5em}\times_{\pi_{N}} TN = \{ (m,v) \in M \times TN \mid \pi_{N}(v) = p(m)\} \qquad \text{and} \\
 & T(p^{*}TN)= TM\, {}_{Tp}\hspace{-.5em}\times_{T\pi_{N}} T^{2}N = \bigl\{ (a,b) \in TM \times T^{2}N \mid Tp(a) = T\pi_{N}(b) \bigr\}.
\end{align*}
Let $\sigma_{\mathcal{H}}\colon p^{*}TN \rightarrow TM$ be section of $(\pi_{M}, {\rm d}p)$ associated to $\mathcal{H}$.
We define
\begin{equation*}
 \hat{S}\colon\ p^{*}TN \rightarrow T(p^{*}TN), \qquad (m,v) \mapsto (\sigma_{\mathcal{H}}(m,v), S(v)).
\end{equation*}
Because $S$ is a spray and $\sigma_{\mathcal{H}}$ is a section of $(\pi_{M},{\rm d}p)$, we have for $(m,v) \in M \, {}_{p}\hspace{-.25em}\times_{\pi_{N}} TN$,
\begin{align*}
 T\pi_{N}(S(v)) = v, \qquad
 Tp(\sigma_{\mathcal{H}}(m,v)) = v,
\end{align*}
respectively.
As a consequence, $\hat{S}$ is a spray on the anchored bundle $(p^\ast\pi_N, p^\ast TN, \sigma_{\cH})$.
Now, we recall that the corestriction of $\sigma_\cH \colon p^\ast TN \rightarrow TM$ to $\mathcal{H} \subseteq TM$ is a diffeomorphism with inverse $(\pi_M,{\rm d}p)|_\cH$.
We thus transfer the vector field $\hat{S}$ to a vector field $S_{\mathcal{H}}\colon \mathcal{H} \rightarrow T\mathcal{H}$, explicitly given~by\looseness=-1
\begin{displaymath}
 S_\cH \colon \ \cH \rightarrow T\cH,\qquad S_\cH = T\sigma_\cH \circ \hat{S} \circ (\pi_M, {\rm d}p)|_{\cH}.\end{displaymath}
Applying Lemma~\ref{lem:anchoredIso}, $S_{\cH}$ is a spray on the anchored bundle $(\pi_{\cH}, \cH, \iota_\cH)$, where $\iota_\cH \colon \cH \rightarrow TM$ is the subbundle inclusion.
Note now that on $\cH$ we have $Tp|_\cH = (\pi_N^\ast p) \circ (\pi_M,{\rm d}p)|_\cH$ and thus by the chain rule
\begin{align*}
(T(Tp|\cH) \circ S_\cH &= T(Tp \circ \sigma_\cH) \circ p^\ast S\circ (\pi_M, {\rm d}p)|_{\cH} = T(\pi_N^\ast p) \circ p^\ast S \circ (\pi_M, {\rm d}p)|_{\cH} \\
&= S \circ (\pi_N^\ast p) \circ (\pi_M, {\rm d}p)|_{\cH} = S \circ Tp|_{\cH}.
\end{align*}
So $S$ is $(Tp|_{\cH})$-related to the vector field $S_\cH$.

The final statement on the exponential now follows directly from the fact that the exponential is constructed as the projection of the flow of vector fields (at time $1$). By relatedness of the vector fields, integral curves of $S_\cH$ get mapped by $Tp|_{\cH}$ to integral curves of $S$, hence the relation for the exponentials holds.
\end{proof}

\section{Proof of the Stacey--Roberts lemma}

In this section, we establish the main theorem, namely that the pushforward by a submersion becomes a submersion between the manifolds of mappings.

{\bf Plan of proof for Theorem~\ref{TheoremA}.}\label{plan_of_attack}
Let $p \colon M \rightarrow N$ be a surjective submersion between $C^\infty$-paracompact Banach manifolds. For any $\sigma$-compact manifold $X$, we would like to prove that $p_\ast \colon C^\infty (X,M) \rightarrow C^\infty (X,N)$ is a submersion, i.e., that there are submersion charts for $p_\ast$.
The~manifold structure on $C^\infty (X,M)$ and $C^\infty (X,N)$ is determined by an atlas of canonical charts constructed from local additions on $M$ and $N$.

{\bf Idea.} \emph{Construct $\Sigma_M$, $\Sigma_N$ such that the canonical charts of the manifolds of mappings become submersion charts for the pushforward by the submersion.}

To follow the general idea, we need to pick a horizontal complement to the vertical bundle and local additions $\Sigma_M$, $\Sigma_N$ such that the following diagram commutes:
\begin{equation}\label{diag:SR-Lemma}
\begin{tikzcd}
 TM=\mathcal{V} \oplus \mathcal{H} \arrow[d, ,"0\oplus Tp|_{\mathcal{H}}"] & \arrow[l,hook'] \Omega_M \arrow[rr,"\Sigma_M"] \arrow[d,"Tp|_{\Omega_M}"] & &M \arrow[d,"p"]\\ TN & \arrow[l, hook'] \Omega_N \arrow[rr,"\Sigma_N"] & &\,N.
\end{tikzcd}
\end{equation}
For these local additions, the canonical charts of the manifold of mappings split as submersion charts.
We have already prepared all the ingredients required for the construction of the elements of \eqref{diag:SR-Lemma}, indeed we proceed in steps:
\begin{enumerate}\itemsep=0pt
\item Pick an Ehresmann connection $\cH$ for $p$, Lemma~\ref{lem:ex_Ehresmann}.
\item Pick a spray $S_N$ on $N$ whose exponential $\exp_{S_N}$ yields the local addition $\Sigma_N$, Corollary~\ref{cor:exp_locadd}. Denote the domain of $\Sigma_N$ by $\Omega_N$.
\item Construct from $S_N$ a $Tp$ related spray $S_\cH$ on the horizontal bundle $\cH$, Lemma~\ref{lem:spray_lifting}. Then the exponentials satisfy $p \circ \exp_{S_\cH} = \exp_{S_N} \circ Tp|_\cH = \Sigma_N \circ Tp|_{\cH}.$
\item Pick a local addition $\Psi=(\pi_{\cV},\exp_{S_{\cV}})$ on the submersion $p$, Lemma~\ref{lem:con_loc_add_subm}\,(b).
\item Pick a linear connection, Remark~\ref{rem:ex_lincon}, on $\cV$ and denote its parallel transport (cf.\ page~\pageref{setup:parallel_tp}) along a smooth curve $\gamma \colon [0,1] \rightarrow M$ by $P^\gamma_{0,1}\colon \cV_{\gamma(0)} \rightarrow \cV_{\gamma(1)}$.
\item Construct $\Sigma_M$ by the following recipe. In Proposition~\ref{prop:custom_sigma} below, we construct an open neighborhood of the zero section $\Sigma_M \subseteq TM$ such that the following construction is well defined for $(v,h) \in \Omega_M \subseteq \cV \oplus \cH$:
\begin{enumerate}\itemsep=0pt
\item parallel transport $P_{0,1}^{c_h}(v)$, $v\in \cV$ along $c_h(t)=\exp_{S_\cH}(th)$, where $h \in \cH$,
\item then apply $\exp_{S_p}$ to the resulting vertical vector.
\end{enumerate}
\end{enumerate}

What remains to be shown is that the map $\Sigma_{M}$ constructed in~(6) actually yields a local addition (Proposition~\ref{prop:custom_sigma}), and that the charts obtained by using $\Sigma_{N}$ and $\Sigma_{M}$ are submersion charts.
For this, we establish smoothness of the parallel transport.

\begin{Lemma}\label{lem:smooth_aux}
In the situation of Step~{\rm 5}, there exists an open, fibre-wise convex neighborhood ${U \subseteq TM}$ of the zero section, such that map
\begin{displaymath}
\rho \colon \ \cV \oplus \cH \supseteq U \rightarrow TM, \qquad \rho (v,h) := P_{0,1}^{c_h(t)}(v) \qquad \text{with } c_h (t) = \exp_{S_{\cH}} (th)
\end{displaymath}
makes sense and is smooth.
\end{Lemma}

\begin{proof}
 We shall prove slightly more and show that $P_{r,s}^{c_h}$ is smooth in all its arguments.
On its domain $D\subseteq \cH$, the exponential map $\exp_{S_{\cH}}$ is smooth. Now pick an open zero section neighborhood $U\subseteq TM$ fibre-wise convex such that $U \cap \cH \subseteq D$. Note that with this choice of~$U$ the map $\rho$ makes sense.

As smoothness is a local property, we pick an open neighborhood $U_h$ of $h \in \cH \cap U$ and dissect the compact set $[0,1]$ into finitely many subintervals, such that for all $x\in U_h$ and all~${t \in J}$ in a subinterval, the curves $\exp_{S_\cH}(tx)$ are contained in a single manifold chart. Shrinking~${U_h}$, we~may assume that it is the domain of a bundle trivialisation. Thus we can without loss of generality work in charts and bundle trivialisations which we suppress in the notation. Denote by $B_V \colon V \rightarrow \text{Bil}(E , F; F)$ the smooth local map associated to the linear connection, Definition~\ref{defn:connect_lin}, where we assume that $\exp_\cH (tx)\in V$ for all $(x,t) \in U_h\times J$. By construction, the map
\begin{displaymath}F \colon\ U_h \times J \times F \rightarrow F, \qquad F(x,t,v) =- B_V\bigl(\exp_{S_{\cV}}(tx)\bigr)\biggl(\frac{{\rm d}}{{\rm d}t}\exp_{S_{\cV}}(tx),v\biggr)
\end{displaymath}
is smooth. Now parallel transport is the solution to the differential equation \eqref{eq:PT_diffeo}, which in our case is equivalent to
\begin{displaymath}
 \dot{\eta}(t) = F(h,t,\eta(t)), \qquad \eta(0)=v.
\end{displaymath}
Thus the assertion follows from smooth parameter and initial value dependence of differential equations on (open subsets) of Banach spaces, see \cite[Chapter~IV, Section~1]{La01}.
\end{proof}

\begin{Proposition}\label{prop:custom_sigma}
We use notation as in 
Steps~$1$--$5$~and use the splitting $TM=\cV \oplus \cH$ to split $x_m=(v_m,h_m) \in T_m M$ into a vertical and a horizontal component.
There exists an open neighborhood $\Omega_M \subseteq TM$ of the $0$-section such that the map
\begin{align}\label{evil_loc_add}
\Sigma_M\colon \ TM\supseteq \Omega_M \rightarrow M,\qquad \Sigma_{M} (v,h) := \exp_{S_{\cV}}\bigl(P_{0,1}^{c_h(t)}(v)\bigr) \qquad \text{with } c_h(t)=\exp_{\cH}(th)\!\!
\end{align}
defines a local addition which makes \eqref{diag:SR-Lemma} commutative.
\end{Proposition}

\begin{proof}
We define $\Sigma_M := (0 \oplus Tp|_{\cH})^{-1}(\Sigma_N) \cap \rho^{-1} (O)$, where $\rho$ is the map from Lemma~\ref{lem:smooth_aux} and~${O \subseteq \cV}$ the domain of $\exp_{S_p}$. Then $\Omega_M$ is an open zero-section neighborhood and by definition~$\Sigma_M$ makes sense on $\Omega_M$. Further, the map $\Sigma_M$ is smooth as a composition of smooth mappings using Lemma~\ref{lem:smooth_aux}. Note that $\Sigma_M (0_m) = m$ for every $m \in M$ as $\exp_{S_\cH} (0_m)=m$, parallel transport along the constant path is the identity and \smash{$\Psi=\bigl(\pi_{\cV}, \exp_{S_\cV}\bigr)$} is a local addition~on~$p$.

{\it Step $1$:} \emph{$\Sigma_M$ makes \eqref{diag:SR-Lemma} commutative.} We insert the formula~\eqref{evil_loc_add} for $\Sigma_M$ to obtain
\begin{align*}
p\circ \Sigma_M (v_x,h_x)&= p\circ \Psi \bigl(P_{0,1}^{c_{h_x}(t)}(v_x)\bigr) \stackrel{\eqref{diag:loc_add_on_sub}}{=} p \bigl(\exp_{S_\cH} (h_x)\bigr)\stackrel{\text{Lemma}~\ref{lem:spray_lifting}}{=} \exp_{S_N} \circ Tp|_\cH (h_x)\\
& = \exp_{S_N} \circ (0\oplus Tp|_\cH)(v_x,h_x) = \Sigma_N\circ (0\oplus Tp|_\cH)(v_x,h_x).
\end{align*}
We will now compute the tangent map of $(\pi_M,\Sigma_M)$ at the zero section to be able to apply the open mapping theorem and obtain a diffeomorphism onto the a diagonal neighborhood~in~${M\times M}$.

{\it Step $2$:} \emph{Splitting and localisation.}
We exploit the splitting $TM = \cV \oplus \cH$ to compute the~derivative. Assume that $\alpha \colon {}]{-}\varepsilon,\varepsilon[ \rightarrow TM$ is smooth with $\alpha(0)=0_m$. Split $\alpha(t)=\alpha_\cV(t) + \alpha_\cH(t)$,~as
\begin{displaymath}
 \alpha_\cH := \sigma_\cH \circ (\pi,{\rm d}p) \circ \alpha, \qquad \alpha_\cV := \alpha - \alpha_\cH,
\end{displaymath}
where $\sigma_\cH$ is the section of $(\pi,{\rm d}p)$ whose image is the horizontal subbundle. In particular, $\alpha_\cH$ takes its image in $\cH$ while $\alpha_\cV$ takes its image in $\cV$. Further, we set $\alpha_M := \pi_{M} \circ \alpha$ for the base curve of $\alpha$ in $M$. The derivative will be computed locally in the next step.

For $m \in M$ fixed, pick a chart domain $m\in U$ and bundle trivialisations for $\cV|_U\cong U \times F$ and $\cH|_U \cong U \times H$, where $E=F\oplus H$ for the fibres of the subbundles and $E$ the modelling space of~$M$. In the following, we shall suppress the trivialisations in the formulae, but note that
\begin{displaymath}
 T_{0_m}TM|_U \cong E \times E, \qquad \text{and}\qquad T_m M \times T_m M\cong E \times E.
\end{displaymath}
With respect to this identification, we rewrite $T_{0_x} (\pi_M,\Sigma_M) \colon T_{0_m} TM \rightarrow T_m M \times T_{m} M$ as a block matrix of linear operators,
\begin{align}\label{eq:blockoperator}
\underbrace{E \times E}_{\cong T_{0_m}TM|_U} \rightarrow \underbrace{E\times E}_{\cong T_mM\times T_mM}, \qquad T_{0_m} (\pi,\Sigma_M) = \biggl[\begin{NiceTabular}{c|c}
$\id_E$ & $0$ \\ \hline
$\star$ & $T_{0_m} \Sigma_M (0,\cdot)$
\end{NiceTabular}\biggr].
\end{align}
Here $T_{0_m} \Sigma_M (0,\cdot)$ coincides with the derivative in the fibre $T_m M$ of $\Sigma_m|_{T_mM }$. We shall now compute these fibre derivatives.

{\it Step $3\,(a)$:} \emph{Fiber derivative of $\Sigma_M$ on the subspace $\cV_m$.} Since $\alpha_\cV$ takes its image in the vertical bundle, we note that \smash{$\exp_{S_\cH} \bigl(c_{\alpha_\cV(t)}(s)\bigr)= \exp_{S_\cH}\bigl(s0_{\alpha_M (t)}\bigr) =\pi_{\cV}(s\alpha_{\cV}(t))=\alpha_M(t)$} is constant in~$s$. Thus parallel transport along this path reduces to the identity (as the path is constant):
\begin{align*}
T_{0_m} (\pi_M,\Sigma_M)\bigl(\alpha_\cV'(0)\bigr) &= \frac{{\rm d}}{{\rm d}t}\biggr|_{t=0} (\pi_M,\Sigma_M)(\alpha_\cV (t)) = \frac{{\rm d}}{{\rm d}t}\biggr|_{t=0} (\alpha_M (t), \exp_{S_{\cV}} (\alpha_\cV (t))) \\
&= \frac{{\rm d}}{{\rm d}t}\biggr|_{t=0} \Psi (\alpha_{\cV} (t)) =T_{0_m}\Psi (\alpha_{\cV}'(0)).
\end{align*}
The map $\Psi = (\pi_{\cV},\exp_{S_{\cV}})$ is an embedding from an open neighborhood of the zero section in $\cV$ into $M \times_p M$ and by construction $T_{0_m} \exp_{S_{\cV}}|$ takes its image in $\cV_m$. Since we are after the fibre derivative of $\Sigma_M$, we can specialise to a vertical curve $\alpha_{\cV}(t)=tv$ for some $v\in \cV_m$. Identifying the fibre derivative of the second component of $\Psi$ via \eqref{eq:TG2_deriv}, we see that
\[T_{0_m}\Sigma_M|_{T_m M \cap \cV} (v)= T_{0_m}\exp_{S_{\cV}}^m (v)=v,\qquad v \in \cV_m.\]

{\it Step $3\,(b)$:} \emph{Fiber derivative of $\Sigma_M$ on the subspace $\cH_m$.} Consider the curve $\alpha_\cH$ from Step~2 which takes only values in the horizontal bundle. Hence, we find for $\Sigma_M$ that
\begin{displaymath}\Sigma_M (\alpha_{\cH}(t)) =\exp_{S_{\cV}}\bigl(P_{0,1}^{\exp_{S_\cH}(s\alpha_{\cH}(t))}(0)\bigr)=\exp_{S_\cH}(\alpha_\cH(t)).
\end{displaymath}
Here the last equality follows from the fact that $\exp_{S_{\cV}}$ is a local addition on $p$ which maps \smash{$0\raisebox{2pt}{${}_{\exp_{S_{\cH}}(\alpha_\cH(t))}$}$} to the base point. Since $\exp_{S_\cH}$ is the exponential of a spray on the anchored bundle~$\cH$, we find with Lemma~\ref{lem:spec_curv} that the fibre derivative takes its image in $\cH$ and computes as~$T_{0_x} \exp_{S_{\cH}}(0,h)=h$, using that the anchor of $\cH$ is the inclusion.

{\it Step $4$:} \emph{$T_0\Sigma_M(0,\cdot)$ is invertible.} From Step 3\,(a) and~(b), we obtain by linearity from the splitting $T_m M = (T_m M \cap \cV) \times (T_mM \cap \cH)$ an identification of $T_{0_x}\Sigma_M (0,\cdot )$ with the block operator
$\bigl[\begin{smallmatrix} \id_{T_m M \cap \cV} & 0 \\ 0 & \id_{T_mM \cap \cH}\end{smallmatrix}\bigr]$, i.e., $T_{0_x}\Sigma_M (0,\cdot ) = \id_{T_m M}$. Inserting into \eqref{eq:blockoperator}, $T_{0_m} (\pi_M,\Sigma_M)$ is invertible for every $m \in M$. Now $M$ is $C^\infty$-paracompact, and $M \times M = M \times_\star M$, where the fibre product is taken with respect to the submersion $M \rightarrow \{\star\}$ onto the one-point manifold. By Lemma~\ref{lem:loc2glob_emb}, we can thus shrink $\Omega_M$ to a smaller zero section neighborhood on which~$(\pi_M,\Sigma_M)$ restricts to an embedding.
\end{proof}

We can now prove Theorem~\ref{TheoremA} using the proof strategy of the original Stacey--Roberts lemma \cite[Lemma~2.4]{AaS19}. The argument is as follows.

\begin{proof}[Proof of Theorem~\ref{TheoremA}]
We already know that the pushforward $p_\ast$ is smooth by Proposition~\ref{prop:smooth_pf}.

Thus fix $f \in C^\infty (X,M)$ and construct submersion charts around $f$ and $p\circ f \in C^\infty (X,N)$. Use Proposition~\ref{prop:custom_sigma} to obtain local additions $\Sigma_M$ on $M$ and $\Sigma_N$ on $N$ such that \eqref{diag:SR-Lemma} commutes. Denote by $\cV$ the vertical bundle induced by $p$ and recall that in the construction a horizontal bundle $\cH$ was chosen such that $TM = \mathcal{V} \oplus \mathcal{H}$. Let $\iota_\cB \colon \cB \rightarrow TM$ and $\text{proj}_{\cB} \colon TM \rightarrow \cB$ be the bundle inclusion, and projection, respectively, of the subbundle $\cB \in \{\cV, \cH\}$.
With respect to these choices, we then construct the canonical charts \eqref{eq:can_charts} of the manifolds of mappings
\begin{displaymath}
 \phi_f^{-1} \colon\ O_f' \rightarrow O_f \subseteq \Gamma_c (f^{\ast}TM) \qquad\text{and}\qquad \phi_{p_\ast (f)}^{-1}\colon\ O_{p_\ast (f)} \rightarrow O_{p_\ast (f)}\subseteq \Gamma_c ((p\circ f)^\ast TN).
\end{displaymath}
Further, we identify now $\Gamma_c (f^\ast TM)=\Gamma_c (f^\ast (\cV\oplus \cH))$ as follows: There exists a natural bundle isomorphism $\kappa =(\kappa_\cV, \kappa_\cH) \colon f^\ast (\cV \oplus \cH) \rightarrow f^\ast \cV \oplus f^\ast \cH$. Applying Gl\"{o}ckner's $\Omega$-lemma \mbox{\cite[Theorem~F.23]{glo04}} (which is applicable since the pullback bundles are over the $\sigma$-compact base manifold~$X$, see \cite[Lemma F.19\,(c)]{glo04}), we obtain a vector space isomorphism
\begin{align}\label{eq:hor_vert_split}
((f^\ast \text{proj}_\cV)_\ast, (f^\ast \text{proj}_\cH)_\ast) \circ \kappa_\ast \colon\ \Gamma_c (f^\ast (\cV\oplus \cH))\rightarrow \Gamma_c(f^\ast \cV)\times \Gamma_c(f^\ast \cH).
\end{align}
(whose inverse is given by $ (X,Y) \mapsto \kappa^{-1}_\ast ((f^\ast \iota_\cV)_\ast (X) + (f^\ast \iota_{\cH})_\ast (Y))$ hence it is smooth by another application of Gl\"{o}ckner's $\Omega$-lemma).
By the universal property of the pullback, we obtain a~bundle morphism over the identity
\begin{displaymath}
 \Theta \colon\ f^\ast TM \rightarrow (p_\ast(f))^\ast TN,\qquad \Theta (v)=(\pi_M,{\rm d}p)(v).
\end{displaymath}
On the subbundle $\iota_{\cH}\colon f^\ast \rightarrow f^\ast TM$, the map $\Theta$ restricts to a bundle isomorphism $B\colon f^\ast \cH \rightarrow (p_\ast (f))^\ast TN$. Define now $I_f := \iota_{\cH} \circ B^{-1}$. Using the formula for the parametrisations~$\phi_f$,~$\phi_{p\circ f}$ from \eqref{eq:can_charts}, we can combine \eqref{diag:SR-Lemma}, \eqref{eq:hor_vert_split} to obtain the following commutative diagram:
\begin{displaymath}
\begin{tikzcd}
\Gamma_c (f^\ast (\cH \cap \Omega_M)) \arrow[r,"(\iota_{\cH})_\ast"] & \Gamma_c (f^\ast (\Omega_M)) \arrow[r, "\phi_f"] \arrow[d,"{(\pi_M ,{\rm d}p)_\ast}"] & C^\infty (X,M) \arrow[d, "p_\ast"] \\
\Gamma_c ((p_\ast(f))^\ast \Omega_N) \arrow[u,"B^{-1}_\ast"] \arrow[r,equal] \arrow[ru,"(I_f)_\ast"]& \Gamma_c (p \circ f)^\ast \Omega_N) \arrow[r,"\phi_{p_\ast (f)}"] &
C^\infty (X,N).
\end{tikzcd}
\end{displaymath}
Hence in these canonical charts the map $p_\ast$ satisfies
\begin{displaymath}
\phi^{-1}_{p_\ast(f)} \circ p_\ast \circ \phi_f = ((\pi_M,{\rm d}p)|_{\Omega_M})_\ast.
\end{displaymath}
We note that $(\pi_M,{\rm d}p)_\ast$ is continuous linear with continuous linear right inverse $(I_f)_\ast$ on the spaces of sections. In other words, the canonical charts are submersion charts in the sense of Definition~\ref{defn:submersion}, showing that $p_\ast$ is a submersion.
\end{proof}

\begin{Remark} \quad
\begin{enumerate}\itemsep=0pt
\item As already remarked in \cite{AaS19} for the finite-dimensional case, even though $p$ is a~surjective submersion, $p_\ast$ is not necessarily surjective. For an example, recall from \cite{dH16} that the Ehresmann connection giving the horizontal bundle $\cH$ need not be complete. Hence for~${X=\R}$ not every smooth curve can be globally lifted to $M$ for such a connection.

\item The manifold $X$ did not play a direct role in the proof, so the proof remains valid for manifolds with boundary or corners (or even more generally for manifolds with rough boundary, cf.\ \cite[\S2.3]{GaS22} for the definition and further explanations). However, several results we cited from the literature would have to be reproved in this more general setting and we refrain from going into the added technical details here.
\end{enumerate}
\end{Remark}

\subsection*{Applications of the generalised Stacey--Roberts lemma}

There are several constructions in the literature where the use of the original Stacey--Roberts lemma restricted the construction to manifolds of mappings with finite-dimensional target. The new version allows one to generalise these results now with the same proof to Banach manifolds. We mention explicitly that
\begin{itemize}\itemsep=0pt
\item The construction of current groupoids $C^\infty (X, \cG)$ from \cite[Theorem~3.2]{AaGaS20} can be carried out now for Banach Lie groupoids, removing the restriction to finite-dimensional target Lie groupoids.
\item The bisection group of a Banach Lie groupoid $\cG =(G\rightrightarrows M)$ over a finite-dimensional base manifold $M$ always is a submanifold of $C^\infty (M,G)$ with the proof as in \cite[Theorem~1.3]{AaS19}.
\end{itemize}
As both direct applications of the Stacey--Roberts lemma involve Banach Lie groupoids we would like to mention that relevant examples have recently been investigated, see \cite{BaGaP19,BanGpd23,BanGpd24,HaJ16} and the references therein.
For the readers convenience we describe the construction of the submanifold structure in more details to showcase a typical application of the Stacey--Roberts lemma.

\textbf{Application example: Submanifolds of manifolds of mappings.}
Let us start with a~general result for a~surjective submersion where we recover a version of \cite[Proposition~10.10]{Mic80} for more general manifolds.

\begin{Lemma}\label{lem:sect_submanf} For a surjective submersion $p \colon M \rightarrow N$ from a $C^\infty$-paracompact Banach manifold~$M$ to a $\sigma$-compact manifold $N$, the space of submersion sections
 \begin{displaymath}
 \Gamma_p := \{\sigma \in C^\infty (N,M) \mid p\circ \sigma = \id_N\}
 \end{displaymath}
is a splitting submanifold.
\end{Lemma}

\begin{proof}
As $p$ is a submersion, the Stacey--Roberts lemma implies that $p_\ast$ is a submersion. Hence $\Gamma_p = (p_\ast)^{-1}(\{\id_N\})$ is a splitting submanifold by the regular value theorem \cite[Theorem D]{glo16}.
\end{proof}

We now return to the bisections of a Banach Lie groupoid $\cG = (G \rightrightarrows M)$, where $G$ is $C^\infty$-paracompact over a $\sigma$-compact base $M$. Denote by $s,t \colon G \rightarrow M$ the source and target map of the groupoid which both are submersions. Then by definition, the bisection group is the set
\begin{displaymath}\Bis (\cG) := \{\sigma \in C^\infty (M,G) \mid s \circ \sigma = \id_M,\, t \circ \sigma \in \Diff (M)\}.\end{displaymath}
Now it is well known that the diffeomorphism group $\Diff (M) \subseteq C^\infty (M,M)$ is an open set, cf.~\cite{Mic80}. Hence $O:= (t_\ast)^{-1}(\Diff (M))$ is open in $C^\infty (M,G)$ by Proposition~\ref{prop:smooth_pf} and we deduce that $\Bis(\cG) = \Gamma_s \cap O$ is a split submanifold of $C^\infty (M,G)$ by Lemma~\ref{lem:sect_submanf}. Alternatively, one could also have noted that the restriction $(s_\ast)|_{O}$ is a submersion as a consequence of the Stacey--Roberts lemma for~$s_\ast$.

With a bit more work, the rest of the proof can be generalised to see that this structure turns the bisection group into a Lie group (the original proof uses results on composition which exist in the literature only for finite-dimensional targets but could be generalised).

The application of the Stacey--Roberts lemma showcased in this section is rather typical as it provides an out of the box argument as to why certain generic constructions (which are available in the infinite-dimensional setting for our notion of submersion, see \cite{glo16} for statements and proofs) on function spaces provide (sub)manifolds. Further examples along this line are provided by the Hom-stack constructions in \cite{RaV18,RaV18}. There a Hom-stack of maps from a manifold into a~Lie groupoid is constructed and shown to be represented by a Lie groupoid. To obtain the Lie groupoid, one however needs to resolved the target groupoid as a \v{C}ech groupoid and construct from the associated covering data a family of fibre-products of mapping spaces. Individually, these fibre-products become manifolds by the usual argument after applying the Stacey--Roberts lemma to the pushforwards used in the construction of the fibre-products.

\appendix

\section{Infinite-dimensional calculus and manifolds}\label{App:Bast}

In this appendix, we recall some basic definitions of infinite-dimensional calculus.
As stated in the beginning, we will mostly work on Banach manifolds, but the manifolds of mappings cannot be modelled on Banach spaces. Hence we need to base our article on a generalised calculus, the so called Bastiani calculus \cite{Bas64}, see, e.g., \cite{GaS22, Sch23} for introductory expositions.

\begin{Definition}
Consider locally convex spaces $E$, $F$
and a map $f\colon U\rightarrow F$ on an open subset $U\subseteq E$.
Write
\[
D_yf(x):= {\rm d}f(x;y):= \frac{{\rm d}}{{\rm d}t}\biggr|_{t=0}f(x+ty)
\]
for the directional derivative of~$f$ at $x\in U$
in the direction $y\in E$, if it exists.
Let $k\in \N_0\cup\{\infty\}$.
If $f$ is continuous, and the iterated directional derivatives
\[
{\rm d}^jf(x;y_1,\dots, y_j):= (D_{y_j}\cdots D_{y_1}f)(x)
\]
exist for all $j\in\N_0$ such that
$j\leq k$, $x\in U$ and $y_1,\dots, y_j\in E$,
and the maps ${\rm d}^jf\colon U\times E^j\to F$
are continuous, then $f$ is called~$C^k$.
A map of class $C^\infty$ is also called~\emph{smooth}.
\end{Definition}

The chain rule holds in the usual form for Bastiani differentiable mappings (see, e.g., \cite[Proposition~1.23]{Sch23}), whence it makes sense to define manifolds as in the finite-dimensional setting via charts. Hence a continuous mapping $f \colon M \rightarrow N$ between (possibly infinite-dimensional, smooth) manifolds is smooth if, locally in suitable pairs of charts the map is Bastiani smooth.

All manifolds and Lie groups considered
in the article are modeled on locally
convex spaces which may be infinite-dimensional,
unless the contrary is stated. If we say that a manifold or a~vector bundle is a Banach manifold or a Banach vector bundle we mean that all modelling spaces (or fibres in the bundle case) are Banach spaces. See~\cite{La01} for a discussion of differential geometry on Banach manifolds. All manifolds are assumed to be smooth and Hausdorff.
We~write~$\pi_M \colon TM \rightarrow M$ for the tangent bundle with bundle projection $\pi_M$.

In the Bastiani setting, the notion of submersion has to be defined as in \cite[Definition 4.4.8]{Ham82}, see also \cite{Sch23}.

\begin{Definition}\label{defn:submersion}
Let $f \colon M \rightarrow N$ be a smooth map between manifolds, we say that $f$ is a~\emph{submersion}, if for every $x \in M$ there exist manifold charts of $M$ around $x$ and of $N$ around $f(x)$ which conjugate the map $f$ to a projection onto a complemented subspace of the model space of~$M$. Pairs of charts with these properties will be called \emph{submersion charts} for the submersion~$f$.\looseness=1
\end{Definition}
We refer to \cite{glo16} for a discussion of the notion of submersion in the infinite-dimensional setting and its properties.
In particular, with this definition of a submersion fibre-products of manifolds can be defined as in the finite-dimensional setting. In particular, the usual construction for a~submersion $p\colon M \rightarrow N$ and a smooth map $g\colon L \rightarrow M$ yields (see \cite[Lemma 1.60]{Sch23}) that the fibre-product
\begin{displaymath}
 M {}_{p}\hspace{-.25em}\times_{g} L := \{(x,y) \in M\times M \mid p(x)=g(y)\}
\end{displaymath}
is a submanifold of $M\times M$. If $g=p$, we also write shorter $M\times_p M$ for this fibre-product.

\begin{Lemma}\label{lem:tangent_fibreprod}
Let $p_i \colon M_i \rightarrow N$, $i=1,2$, be surjective submersions, then the identification $T(M_1\times M_2) \cong TM_1 \times TM_2$ restricts to a diffeomorphism
\begin{displaymath}
 T((M_1) {}_{p_1}\hspace{-.5em}\times_{p_2} (M_2) ) \cong (TM_1) {}_{Tp_1}\hspace{-.5em}\times_{Tp_2} (TM_2).\end{displaymath}
\end{Lemma}

\begin{proof}
This question is clearly local, so we pick submersion charts on $W_i \subseteq M_i$, $i=1,2$, for~$p_i$ such that $W_i \cong U \times V_i$, with $U\subseteq E$, $V_i \subseteq F_i$ open in locally convex spaces and $p_i$ becomes the projection onto the first component with respect to this splitting. Suppressing the submersion charts, the embedding $M_1\times_N M_2 \subseteq M_1 \times M_2$ becomes
\begin{displaymath}
 \iota \colon\ U \times V_1 \times V_2 \rightarrow (U\times V_1) \times (U \times V_2), (u,v_1,v_2) \mapsto ((u,v_1),(u,v_2)).
\end{displaymath}
Passing to the tangent, we obtain the embedding $T\iota$
\begin{align*}
 U \times V_1 \times V_2 \times (E \times F_1) \times (E\times F_2) &\rightarrow ((U\times V_1)\times E\times F_1)\times ((U\times V_2)\times E\times F_2),\\
 (u,v_1,v_2,x,y_1,y_2)&\mapsto ((u,v_1,x,y_1),(u,v_2,x,y_2) ).
\end{align*}
Moreover, as $Tp_i|_{W_i} \colon (U \times V_i) \times E \times F_i \rightarrow U\times E, ((u,v_i),x,y_i) \mapsto (u,x)$, we see that as sets (hence also as manifolds) $T\iota (T((W_1){}_{p_1}\hspace{-.4em}\times_{p_2} \hspace{-.2em} (W_2)) ) ) =(TW_1){}_{Tp_1}\hspace{-.4em}\times_{Tp_2}\hspace{-.2em}(TW_2)$ as claimed.
\end{proof}

\subsection*{Acknowledgements}

The authors wish to thank P.~Steffens who brought the error in the published proof of the Stacey--Roberts lemma to their attention. We thank the referees for their diligent work leading to numerous improvements of the manuscript. Dedicated to D.M.~Roberts on the occasion of his 5th strong prime birthday.

\pdfbookmark[1]{References}{ref}
\LastPageEnding

\end{document}